%% file: pauli.tex
\documentclass[10pt,a4paper]{article}
\usepackage{amsmath,amssymb,amsfonts,amsthm}
\usepackage{graphicx}
\usepackage{float,graphicx,color}
\usepackage[all]{xy}

\input{definitionsSym}

\input{definitionsEnv}

\title{Second order gauge invariant discretizations \\
to the Schr\"odinger and Pauli equations
}

\author{Snorre H. Christiansen\footnote{Department of Mathematics, University of Oslo, PO Box 1053 Blindern, NO 0316 Oslo, Norway}, Tore G. Halvorsen\footnote{Department of Mathematics, University of Oslo, PO Box 1053 Blindern, NO 0316 Oslo, Norway}}
\date{}

\begin{document}

\maketitle

\begin{abstract}
We introduce a numerical method, based on finite elements and lattice gauge theory, to compute approximate solutions to Schr\"odinger and Pauli equations. The crucial geometric property of the method is discrete gauge invariance. The main new achievement is second order convergence. This is proved by interpreting the method as defined on gauge potential dependent finite element spaces and providing an analysis of such spaces in terms of gauge potential dependent norms on simplices of all dimensions. 
\end{abstract}

\section{Introduction}
The Schr\"odinger equation is the fundamental equation of non-relativistic quantum mechanics. It couples to an electromagnetic field through an associated gauge potential (see for instance chapter 4 in \cite{Sha08}). Since there is some freedom in choosing a gauge potential for a given electromagnetic field, it is important that the solution to the Schr\"odinger equation transforms, when the gauge potential changes, in such a way that observable quantities, such as energy levels and probability densities, remain unchanged.

This paper is concerned with developing a finite element method with a similar property, and it is inspired by lattice gauge theory in the sense of \cite{Wil74}. The new method improves upon our previous works \cite{ChrHal11SINUM} and \cite{ChrHal12JMP} by yielding a higher order of convergence. In order to be more precise, on the problem and our new results, we need to introduce some notations.

We let $S$ denote some spatial domain in $\bbR^3$, which we assume to be bounded, convex and polyhedral.

A gauge potential on $S$ is just a vector field. Given one, called $A$, we consider the covariant gradient, defined on complex valued functions on $S$ by:
\begin{equation}
\nabla_A u = \nabla u + i A u. 
\end{equation}

We will be interested in taking scalar products of complex functions, spinors, as well as complexified vectors and one-forms.
The $\bbC$-bilinear scalar products will be denoted:
\begin{equation}
(u,v) \mapsto u \cdot v.
\end{equation}
Hermitian scalar products then take the form:
\begin{equation}
(u,v) \mapsto u^\conj \cdot v,
\end{equation}
where $u^\conj$ denotes the complex conjugate of $u$.

The adjoint of the operator $\nabla_A$ with respect to the $\rmL^2$ hermitian scalar products on $S$ is denoted $\nabla_A^\star$. We may remark that:
\begin{equation}
\nabla_A^\star u = - (\nabla \cdot u + i A \cdot u).
\end{equation}
We then define the covariant Laplacian:
\begin{equation}
\Delta_A  = - \nabla_A^\star \nabla_A.
\end{equation}
We may expand this expression as follows:
\begin{equation}
\Delta_A u = \Delta u + 2i A \cdot \nabla u +i (\nabla \cdot A) u - |A|^2 u.
\end{equation}

We may also define the covariant differential operators:
\begin{equation}
\begin{array}{rlcll}
\grad_A u & = &\nabla_A u & = & \grad u + iA u,\\
\curl_A u &= & \nabla_A \times u &  = & \curl u + i A \times u,\\
\div_A u & = & \nabla_A \cdot u & = & \div u + i A \cdot u,
\end{array}
\end{equation}
in terms of which we have:
\begin{equation}
\Delta_A u = \div_A \grad_A u.
\end{equation}

We are interested in eigenvalue computations of the following form. Find a complex function $u$ on $S$ and $E\in \bbR$ such that:
\begin{equation}\label{eq:eigprob}
- \Delta_A u = E u.
\end{equation}
We use Dirichlet boundary conditions, that is $u|_{\partial S} = 0$. In the last section of the paper we consider some extensions of this eigenvalue problem: inclusion of a scalar potential and taking into account spin. The latter is done through the Pauli equation. However for the remainder of this introduction we stick to (\ref{eq:eigprob}).

The variational formulation is to find $u \in \rmH^1_0(S)$ and $E$ in $\bbR$ such that for all $v \in \rmH^1_0(S)$:
\begin{equation} a(u,v) = E \langle u , v \rangle,
\end{equation}
with:
\begin{equation}
\langle u, v \rangle = \int u^\conj v,
\end{equation}
and:
\begin{equation}
a(u,v) = \int (\nabla_A u)^\conj \cdot (\nabla_A v). 
\end{equation}
Depending on the situation we will include or not the dependence of $a$ on $A$, by writing $a[A]$.

A crucial property of this eigenvalue problem is gauge invariance. Given a scalar field $\alpha$ on $S$, with real values, we may transform $A$ and $u$ as follows:
\begin{align}
A \mapsto A' & = A - \grad \alpha, \label{eq:gaugetransfa}\\
u \mapsto u' & = \exp(i \alpha) u  \label{eq:gaugetransfu}.
\end{align}
Then we have:
\begin{equation}
\nabla_{A'} u' = \exp(i\alpha ) \nabla_A u.
\end{equation}
We therefore have the invariance properties:
\begin{align}
\langle u', v' \rangle & = \langle u, v \rangle, \label{eq:invmass}\\
 a[A'](u',v') & =  a[A](u,v). \label{eq:invstiff} 
\end{align}
These invariance properties are related to local conservation of electric charge, via Noether's theorems.

For the eigenvalue problem, they have the consequence that if:
\begin{equation}
- \Delta_A u = E u,
\end{equation}
then:
\begin{equation}
- \Delta_{A'} u' = E u',
\end{equation}
Moreover, concerning the associated probability densities, we notice that they are the same for $u$ and $u'$. That is, for all $x \in S$ we have:
\begin{equation}
|u'(x)|^2 = |u(x)|^2.
\end{equation}

We wish to construct a numerical method with similar invariance properties.

Consider regular simplicial meshes $\calT_h$ of mesh-width $h$. On $\calT_h$ we have the space $X_h$ of complex-valued continuous piecewise affine functions. The standard Galerkin method is to find $u \in X_h$ and $E$ in $\bbR$ such that for all $v \in X_h$:
\begin{equation} \label{eq:standgal}
a[A](u,v) = E \langle u , v \rangle.
\end{equation}
The standard results on eigenvalue approximation (see \cite{BabOsb89}), show that eigenvectors converge with order $h$ in $\rmH^1$ norm and with order $h^2$ in $\rmL^2$ norm, whereas the eigenvalues converge with order $h^2$.

This method is not gaugeinvariant. We suppose that the field $A$ used in (\ref{eq:standgal}) is a Whitney one form (defined in section \ref{sec:wf}), or equivalently a N\'ed\'elec edge element vector field. It is then natural to consider gauge transformations for $A$ of the form (\ref{eq:gaugetransfa}) with $\alpha$ a Whitney zero form, that is a scalar continuous piecewise affine function. However then the gauge trasnformation (\ref{eq:gaugetransfu}) maps $u$ out of  the space $X_h$.

A solution is to keep gauge transformations (\ref{eq:gaugetransfa}) acting on $A$, but modify the gauge transformations (\ref{eq:gaugetransfa}) acting on $u$, so as to stay within $X_h$:
\begin{align}
A \mapsto A' & = A - \grad \alpha,\label{eq:atransfdisc}\\
u \mapsto u' & = \Pi_h \exp(i \alpha) u \label{eq:utransfdisc},
\end{align}
where $\Pi_h$ is the nodal interpolator onto $X_h$. An interpretation is that one modifies just the nodal values by the gauge trasnformation, but stay piecewise affine. Then the problem is that we no longer have the invariance of the bilinear forms (\ref{eq:invmass}, \ref{eq:invstiff}).

In \cite{ChrHal11SINUM} we proposed a modification of the bilinear forms, inspired by lattice gauge theory \cite{Wil74}\cite{GovUng98}, such that the modified bilinear forms were invariant under the discrete gaugetransformations (\ref{eq:atransfdisc}, \ref{eq:utransfdisc}), yet stayed close to the original ones. In the case of smooth gauge potential we obtained the estimate, for $u,v \in X_h$:
\begin{equation}\label{eq:hconv}
|a(u,v) - \tilde a(u,v)| \cleq h \| u\|_{\rmH^1(S)} \|v\|_{\rmH^1(S)}. 
\end{equation}
This guaranteed that the eigenvectors converge at a rate $h$ in $\rmH^1$. However it seems that the order of convergence in $\rmL^2$ was just $h$, and that the order of convergence for the eigenvalue is also just $h$. Moreover our error estimates were just valid for meshes for which the discrete maximum principle is true. This condition is typically enforced by requiring that dihedral angles be weakly acute.

In \cite{ChrHal12JMP} we proposed, among other things, a more elaborate method for such eigenvalue problems, which did not require such restrictions on the mesh. This was achieved by no longer relying on mass-lumping techniques. However the basic estimate is still (\ref{eq:hconv}), so that the orders of convergence were still just $h$.

\emph{The purpose of this paper is to present and analyse a method which is gauge invariant, but where second order convergence holds. We also include spin in the discussion. Importantly, the theoretical underpinnings of the method, which might be of a broader interest, are of a rather new type.}

The underlying idea is different from our previous works. Even if our method can be interpreted as a variational crime on $X_h$, the analysis relies heavily on interpreting it as a variational crime on another discretization space. This implicit space, denoted $X_h[\tilde A]$, is a discrete space still having one degree of freedom per vertex, but where the local shape of the functions is determined by solving a local PDE related to the global PDE we are addressing. Explicitely we consider functions $u$ such that for all simplexes $T$ in the mesh, of all dimensions, $\Delta_{\tilde A} (u|_{T}) = 0$, where $\tilde A$ is the average of $A$ on $T$. For instance we may remark that $X_h[0] = X_h$ (affine functions on simplices are characterized by the property of being harmonic on all subsimplices). When $\tilde A \neq 0$ we cannot compute explicitely the solutions to these local PDEs, except on edges. As it turns out this will be enough for defining our numerical method.

The idea that the discretization space should incorporate the behavior of the PDE is not new. Multiscale finite element methods are often based on this idea. The method of \cite{XuZik99} can also be interpreted this way, and we share with this method that we are especially concerned with the behavior of discrete functions on edges, whereas the previously mentioned multiscale methods are mostly concerned with the behavior on the maximal simplices (tetrahedra in three space dimensions) often using standard finite elements on the skeleton. We have previously introduced modified shape functions adapted to convection diffusion problems, in the framework of finite element systems \cite{Chr13FoCM}\cite{ChrHalSor14}. It has been our hope that the type of analysis we present here might extend to convection diffusion equations, but so far this har not been realized. For convection diffusion equations one is interested in the regime of vanishing viscosity, whereas in the present case we are \emph{not} considering the highfrequency regime, even though this could be interesting in some experimental setups.

As in the framework of finite element systems, we insist on the recursive nature of our discrete functions: they are defined not only on tetrahedra (top-dimensional simplexes in our mesh) but also on all the subsimplicies of all dimensions. While we have previously been much concerned with the algebra of this recursive structure \cite{Chr08M3AS}\cite{Chr09AWM} (for mixed finite elements or, more generally, differential forms) we have also analysed stable interpolation operators using recursively defined norms \cite{ChrMunOwr11} (e.g. Proposition 5.51). 

In this paper, in section \ref{sec:mainb}, we supply new estimates for recursively defined norms depending on gaugepotentials. This is perhaps our main theoretical novelty. The numerical method is defined and analysed in section \ref{sec:num}. Finally in \ref{sec:pauli} we include scalar potentials and spin, as in the Pauli equation.

\section{Comparisons of recursive norms \label{sec:mainb}}


In this section we study functions defined on simplexes. Given a simplex $T$, we consider functions in $\rmH^1(T)$ with the additional property that the restriction to any face $T'$ of $T$ is in $\rmH^1(T)$. We denote by $\rmH^1_\rec(T)$ the space of such functions, to point out the recursive nature of the construction. One advantage with this space with respect to finite element analysis is that the nodal interpolator (which requires taking vertex values) is well defined independently of space dimension. This reduces the need for appeals to regularity theorems and Sobolev injections.

Our goal is to relate different norms of such functions, in particular norms depending on a choice of gauge potential $A$, defined on $T$ as well as on its faces. For simplicity we consider only the case of gauge potentials which are constant, so that the differential operators $\Delta_A$ have constant coefficients. This has the advantage of guaranteeing for instance that the kernel of $\nabla_A$ is a one-dimensional space of functions.

In what follows $\calA$ will denote a set consisting of constant one-forms (gauge potentials) attached to $T$ and its faces, which is bounded say with respect to the $\rmL^\infty$ norm. In this section, most often one should think of $T$ as a ``reference'' simplex, of diameter of order $1$. If $T$ is a ``physical'' simplex of diameter $h$ and we map it back to a reference simplex $\hat T$ and pull back a gauge potential $A$ on $T$ to $\hat T$, $A$ gets multiplied by $h$. Thus, if $A$ is bounded on $S$, the pullbacks of $A$ to reference elements will indeed live in a bounded set. Notice that the set of constant one-forms on each subsimplex of a given simplex is finite dimensional, therefore $\calA$ will be compact, say in $\rmL^\infty$ norm.

Notationwise we use $T' \subcell T$ to denote that $T'$ is a subsimplex of $T$, $T' \subcellp T$ to say that $T'$ is a subsimplex of $T$ which is not a point, and $T'\subcelln T$ to say that $T'$ is a subsimplex of $T$ distinct from $T$. Given a simplicial complex $\calT$,  $\calT^k$ denotes the set of simplices in $\calT$ of dimension $k$. Similarly $T^0$ denotes the set of vertices of a simplex $T$.

\paragraph{Bounds on $\rmL^2$ norms.}

\begin{proposition}
Let $T$ be a Lipschitz domain.

We have for $u \in \rmH^1(T)$ and $A \in \calA$:
\begin{equation}\label{eq:poinc}
\| u\|_{\rmL^2(T)} \cleq \| \nabla_A u \|_{\rmL^2(T)} + \| u\|_{\rmL^2(\partial T)},
\end{equation}
\end{proposition}

\begin{proof}
We consider first the case $A= 0$. Suppose the inequality does not hold. Let $u_n$ be a sequence in $\rmH^1(T)$ such that:
\begin{align}
\| u_n\|_{\rmL^2(T)} & = 1\\
\| \nabla u_n \|_{\rmL^2(T)} + \| u_n \|_{\rmL^2(\partial T)} & \to 0.
\end{align}
We may extract a subsequence which converges in $\rmL^2(T)$. Since $\nabla u_n$ converges in $\rmL^2(T)$, this subsequence is Cauchy in $\rmH^1(T)$, hence converges. The limit is a constant function which is $0$ on the boundary. But it should also have $\rmL^2$ norm equal to $1$. Since this is impossible, the inequality must hold.

Consider now the possibility that $A \neq 0$. Recall Kato's pointwise a.e. inequality:
\begin{equation}
| \nabla |u|(x)| \leq | \nabla_A u(x)|.
\end{equation}
Then one applies (\ref{eq:poinc}) to $|u|$.
\end{proof}

By induction on dimension we deduce:
\begin{proposition}\label{prop:hzerorec} Let $T$ be a simplex.
 We have, for $u \in \rmH^1_\rec(T)$ and $A \in \calA$:
\begin{equation}
\sum_{T' \subcellp T}\| u\|_{\rmL^2(T)} \cleq \sum_{T' \subcellp T} \| \nabla_A u\|_{\rmL^2(T)} + \sum_{x \in T^0} |u(x)|,
\end{equation}

\end{proposition}

\paragraph{Bounds on $\rmH^1$ norms.}

\begin{proposition}
Let $T$ be a Lipschitz domain.

We have for $u \in \rmH^1(T)$ such that $u|_{\partial T} \in \rmH^1(\partial T)$   and $A \in \calA$:
\begin{equation}
\| u\|_{\rmH^1(T)}    \cleq \| \Delta_A u \|_{\rmH^{-1}(T)} + \| u\|_{\rmH^1(\partial T)}.
\end{equation}
\end{proposition}

\begin{proof}
Suppose not. Choose sequences $u_n$ and $A_n \in \calA$ such that:
\begin{align}
\| u_n\|_{\rmH^1(T)} & = 1,\\
\| \Delta_{A_n} u_n \|_{\rmH^{-1}(T)} + \| u_n\|_{\rmH^1(\partial T)} & \to 0.
\end{align}
We may suppose that $A_n \to A$, since $\calA$ is compact.

Let $v_n$ be an extension to $T$ of $u_n|_{\partial T}$, converging to $0$ in $\rmH^1(T)$. Define $w_n = u_n -v_n \in \rmH^1_0(T)$. We have:
\begin{align}
\| w_n\|_{\rmH^1(T)} & \to 1, \label{eq:honebound}\\
\| \Delta_{A_n} w_n \|_{\rmH^{-1}(T)} & \to 0.
\end{align}

Extract a subsequence of $w_n$ which converges weakly in $\rmH^1(T)$ to, say, $w \in \rmH^1_0(T)$. We have $\Delta_A w = 0$ so $w = 0$. So $w_n$ converges  to $0$, strongly in $\rmL^2(T)$.
We can also write:
\begin{align}
\| \nabla_{A_n} w_n\|^2_{\rmL^2(T)} \leq \| \Delta_{A_n} w_n\|_{\rmH^{-1}(T)} \| w_n\|_{\rmH^1(T)} \to 0.
\end{align}
From which we deduce:
\begin{equation}
\|\nabla w_n\|_{\rmL^2} \cleq \| \nabla_{A_n} w_n \|_{\rmL^2} + \|w_n\|_{\rmL^2} \to 0
\end{equation}
This contradicts (\ref{eq:honebound}).
\end{proof}

By induction we deduce:
\begin{proposition}\label{prop:honerec}
We have estimates, for $u \in \rmH^1_\rec(T)$ and $A \in \calA$:
\begin{equation}
\sum_{T' \subcellp T}\| u\|_{\rmH^1(T)} \cleq \sum_{T' \subcellp T} \| \Delta_A u\|_{\rmH^{-1}(T')} + \sum_{x \in T^0} |u(x)|.
\end{equation}

\end{proposition}

\paragraph{Bounds on $\rmH^2$ norms}

\begin{lemma}\label{lem:htwoext}
Any element of $ \in \rmH^2_\rec(\partial T)$ has an extension to $T$ which is in $\rmH^2(T)$. More precisely, a linear bounded extension operator $\rmH^2_\rec(\partial T) \to \rmH^2(T)$ can be constructed.
\end{lemma}
\begin{proof}
The existence of the extension can be shown exactly as in the proof of Proposition 3.3 in the preprint of \cite{Chr07NM}. 

The first step is to show that, for any face $T'$ of $T$, if $v \in \rmH^2(T') \cap \rmH^1_0(T')$ it can be extended first to an $\rmH^2$ function on the affine space spanned by $T'$ and then extended to all of $T$ by pullback and cut-off, in such a way that the extension is $0$ on all faces of $T$ with dimension $\dim T'$, except of course $T'$ itself. 

The second step is a recursive construction of the extensions, detailed in Proposition 2.2 in \cite{ChrRap13}.
\end{proof}

\begin{proposition}\label{prop:htworec}
Let $T$ be a simplex. For $u \in \rmH^2_\rec(T)$ and $A \in \calA $ we have an estimate:
\begin{equation}
\| u\|_{\rmH^2(T)} \cleq \| \Delta_A u \|_{\rmL^2(T)} + \| u\|_{\rmH^2_\rec(\partial T)}.
\end{equation}

\end{proposition}
\begin{proof}
Suppose now that the estimate does not hold. Choose a sequence $u_n$ in $\rmH^2_\rec(T)$ such that:
\begin{align}
\| u_n\|_{\rmH^2(T)} & = 1,\\
\| \Delta_{A_n} u_n \|_{\rmL^2(T)} + \| u_n \|_{\rmH^2_\rec(\partial T)} & \to 0.
\end{align}
Using the preceding proposition, choose an extension $v_n$ of $u_n|_{\partial T}$ to $T$ such that:
\begin{equation}
\| v_n\|_{\rmH^2(T)} \to 0.
\end{equation}

Setting $w_n = u_n - v_n \in \rmH^2(T) \cap \rmH^1_0(T)$ we have:
\begin{align}
\| w_n\|_{\rmH^2(T)} & \to 1,\label{eq:wntoone}\\
\| \Delta_{A_n} w_n \|_{\rmL^2(T)} & \to 0.
\end{align}

We have:
\begin{equation}
\| \nabla_{A_n} w_n \|_{\rmL^2(T)}^2 \leq \| \Delta_{A_n} w_n \|_{\rmL^2(T)} \, \| w_n\|_{\rmL^2(T)} \to 0,
\end{equation}
Hence $w_n$ converges to $0$ in $\rmH^1(T)$. We then have:

\begin{align}
\| \Delta w_n \|_{\rmL^2(T)} \leq&  \| \Delta_{A_n} w_n \|_{\rmL^2(T)} +\\
& \quad (2 \| A_n\|_{\rmL^\infty} + \|A_n\|_{\rmL^\infty}^2 + \|\div A_n\|_{\rmL^\infty})\| w_n\|_{\rmH^1(T)},\\
 \to&  0.
\end{align}

By elliptic regularity we deduce that $w_n$ converges to $0$ in $\rmH^2(T)$. This contradicts (\ref{eq:wntoone}).

\end{proof}

\begin{remark} Proposition \ref{prop:htworec} does not hold on arbitrary cells. It fails for instance for cells that have several adjacent faces which are coplanar.
\end{remark}

\begin{remark} Lemma \ref{lem:htwoext} and its proof can be extended verbatim to $\rmH^k_\rec$ for $k \geq 2$. However we only need it in conjunction with Proposition \ref{prop:htworec}, which is limited to the case $k=2$, since elliptic regularity on simplexes does not give $\rmH^3$ estimates.
\end{remark}

By induction we deduce:
\begin{proposition}\label{prop:bhhtwo}
Suppose  $u \in \rmH^1_\rec(T) $ satisfies, for each $T' \subcellp T$, $\Delta u \in \rmL^2(T')$. Then $u \in \rmH^2_\rec(T)$ and we have an estimate: 

\begin{equation}
\sum_{T' \subcellp T} |u|_{\rmH^2(T)} \cleq \sum_{T' \subcellp T} \| \Delta_A u \|_{\rmL^2(T)} + \sum_{x \in T^0} |u(x)|.
\end{equation}
\end{proposition}

\paragraph{More bounds.}

In this paragraph we are interested in bounding various quantities by a recursive norm based only on the covariant Laplacian and norms attached to the whole simplex $T$.

\begin{proposition}\label{prop:pointbound} For $u \in \rmH_\rec^1(T)$ and $A\in \calA$, we have an estimate:
\begin{equation}
 \sum_x |u(x)| \cleq \sum_{T' \subcellp T} \| \Delta_A u \|_{\rmH^{-1}(T')} + \|u\|_{\rmL^2(T)} .
\end{equation}
\end{proposition}
\begin{proof}
If not, choose $u_n$ so that:
\begin{align}
 \sum_x |u_n(x)| & = 1, \label{eq:ueqone}\\
\sum_{T' \subcellp T} \| \Delta_{A_n} u \|_{\rmH^{-1}(T')} + \|u\|_{\rmL^2(T)} & \to 0.
\end{align}
The sequence is bounded in $\rmH^1_\rec(T)$, by Proposition \ref{prop:honerec}. Extract a subsequence which converges weakly in $\rmH^1_\rec(T)$. The limit must be $0$. This contradicts (\ref{eq:ueqone}), because, on edges, the trace operator is compact from $\rmH^1(T')$ to vertex values.
 \end{proof}

\begin{proposition}\label{prop:ltwobound}
Let $T$ be a simplex.

We have the following estimate. For all $A\in \calA$ and  $u \in \rmH^1_\rec(T)$ we have:

\begin{equation}
\sum_{T' \subcell T} \|  u \|_{\rmL^2(T')} \cleq \sum_{T' \subcellp T} \|\Delta_A u\|_{\rmH^{-1}(T')} + \| u \|_{\rmL^2(T)} . 
\end{equation}
\end{proposition}
\begin{proof}
We simply write, for any subsimplex $S$ of $T$:
\begin{align}
\|  u \|_{\rmL^2(S)}  & \cleq \sum_{T' \subcellp S} \|\Delta_A u\|_{\rmH^{-1}(T')} + \sum_{x \in S^{0}} |u_x|,\\
& \cleq \sum_{T' \subcellp T} \|\Delta_A u\|_{\rmH^{-1}(T')} + \sum_{x \in T^{ 0}} |u_x|, \\
& \cleq \sum_{T' \subcellp T} \|\Delta_A u\|_{\rmH^{-1}(T')} + \| u \|_{\rmL^2(T)}, 
\end{align}
using Propositions \ref{prop:honerec} and \ref{prop:pointbound}.
\end{proof}

\begin{proposition}\label{prop:gradbound} 
Let $T$ be a simplex. For $u \in \rmH^1_\rec(T)$ and $A \in \cal A$ we have a bound:
\begin{equation}
\sum_{T' \subcellp T} \| \nabla u\|_{\rmL^2(T')} \cleq \sum_{T' \subcellp T} \| \Delta_A u \|_{\rmH^{-1}(T')} + \| \nabla u\|_{\rmL^2(T)} + \|A\|_{\rmL^\infty} \|u\|_{\rmL^2(T)}.
\end{equation}
\end{proposition}
\begin{proof}
Suppose that $u_n \in \rmH^1_\rec(T)$ and $A_n \in \calA$ are sequences such that:
\begin{align}
\sum_{T' \subcellp T} \| \nabla u\|_{\rmL^2(T')} & = 1, \label{eq:nablaone}\\
\sum_{T' \subcellp T} \| \Delta_{A_n} u_n \|_{\rmH^{-1}(T')} + \| \nabla u_n\|_{\rmL^2(T)} + \|A_n\|_{\rmL^\infty} \|u_n\|_{\rmL^2(T)} & \to 0.
\end{align}

Suppose first that:
\begin{equation}
\liminf_n \| A_n\|_{\rmL^\infty} >0.
\end{equation}
Then $u_n$ converges to $0$ in $\rmL^2(T)$, and we get a contradiction using Propositions \ref{prop:pointbound} and \ref{prop:honerec}.

Suppose, on the other hand, that we can extract a subsequence such that $\| A_n\|_{\rmL^\infty} \to 0$. Let $\tilde u_n$ be the average of $u_n$ on $T$. 

We then have:
\begin{equation}
\| u_n -\tilde u_n \|_{\rmH^1(T)} \to 0.
\end{equation}
We also have, for any face $T'$ of $T$:
\begin{align}
\| \Delta_{A_n} (u_n - \tilde u_n) \|_{\rmH^{-1}(T')} & \leq \| \Delta_{A_n} u_n \|_{\rmH^{-1}(T')} + \|\div_A (A \tilde u_n)\|_{\rmH^{-1}(T')},\\
& \cleq \| \Delta_{A_n} u_n \|_{\rmH^{-1}(T')} + \|A_n \tilde u_n \|_{\rmL^2(T')}\\
&  \cleq \| \Delta_{A_n} u_n \|_{\rmH^{-1}(T')} + \|A_n\|_{\rmL^\infty} \| \tilde u_n \|_{\rmL^2(T)},\\
& \to 0.
\end{align}
It follows that $u_n - \tilde u_n$ converges to $0$ in $\rmH^1_\rec(T)$, which contradicts
(\ref{eq:nablaone}).

This concludes the proof.
\end{proof}

\begin{proposition}\label{prop:htwobound}
Let $T$ be a simplex. For $u \in \rmH^2_\rec(T)$ and $A \in \cal A$, we have an estimate:

\begin{equation}
\sum_{T' \subcell T} |  u |_{\rmH^2(T')} \cleq  \sum_{T' \subcellp T} \| \Delta_A u\|_{\rmL^2(T)} + \| A\|_{\rmL^\infty} \| \nabla u \|_{\rmL^2(T)} + \| A\|_{\rmL^\infty} ^2 \| u \|_{\rmL^2(T)} . 
\end{equation}

\end{proposition}
\begin{proof}
We proceed as in the preceding proof.

Suppose that $u_n \in \rmH^1_\rec(T)$ and $A_n \in \calA$ are sequences such that:
\begin{align}
\sum_{T' \subcellp T} |  u |_{\rmH^2(T')} & = 1, \label{eq:hessone}\\
\sum_{T' \subcellp T} \| \Delta_{A_n} u_n \|_{\rmL^{2}(T')} + \|A_n\|_{\rmL^\infty} \| \nabla u_n\|_{\rmL^2(T)} + \|A_n\|_{\rmL^\infty}^2 \|u_n\|_{\rmL^2(T)} & \to 0.
\end{align}

Suppose first that:
\begin{equation}
\liminf_n \| A_n\|_{\rmL^\infty} >0.
\end{equation}
Then we get a contradiction from Proposition \ref{prop:htworec}.

Suppose, on the other hand, that we can extract a subsequence such that $\| A_n\|_{\rmL^\infty} \to 0$. Let $\tilde u_n$ be the affine function which best approximates $u_n$ in $\rmH^1$ norm. We have:
\begin{equation}
\| u_n - \tilde u_n\|_{\rmH^2(T)} \cleq |u_n|_{\rmH^2(T)} \leq 1.
\end{equation}

We also have:
\begin{align}
& \| \Delta_{A_n} (u_n - \tilde u_n) \|_{\rmL^2(T')}\\  & \leq \| \Delta_{A_n} u_n \|_{\rmL^2(T')} +  \| \Delta_{A_n} \tilde u_n \|_{\rmL^2(T')},\\
 & \leq \| \Delta_{A_n} u_n \|_{\rmL^2(T')} + \|A_n\|_{\rmL^\infty} \| \nabla \tilde u_n\|_{\rmL^2(T')} + \|A_n\|_{\rmL^\infty}^2 \| \tilde u_n\|_{\rmL^2(T')},\\
& \leq \| \Delta_{A_n} u_n \|_{\rmL^2(T')} + \|A_n\|_{\rmL^\infty} \| \nabla \tilde u_n\|_{\rmL^2(T)} + \|A_n\|_{\rmL^\infty}^2 \| \tilde u_n\|_{\rmL^2(T)},\\
& \to 0. \label{eq:deltatozero}
 \end{align}

It follows that $u_n - \tilde u_n$ is bounded in $\rmH^2_\rec(T)$. Extract a weakly converging subsequence. The limit must be harmonic on every subsimplex, hence affine. It is also orthogonal to affine functions, hence $0$. Therefore the vertex values of $u_n - \tilde u_n$ converge to $0$. Combined with (\ref{eq:deltatozero}) this shows that $u_n - \tilde u_n$ converges to $0$ in $\rmH^2_\rec(T)$, contradicting (\ref{eq:hessone}).

\end{proof}

\section{Numerical method \label{sec:num}}

\subsection{Whitney forms \label{sec:wf}}

Given a simplicial mesh, the barycentric coordinate map associated with vertex $i$ is denoted $\lambda_i$. It is continuous, piecewise affine, has value $1$ at vertex $i$ and $0$ at the other ones, and is uniquely determined by these properties. 

We assume that all simplices in our mesh have been oriented. Given a simplex $T$ of dimension $k$, with vertices numbered from $0$ to $k$, the associated Whitney form is:
\begin{equation}
\lambda_T = k! \sum_{j = 0}^k (-1)^j \lambda_j \rmd \lambda_0 \wedge \ldots \widehat {\rmd \lambda_j } \ldots \wedge \rmd \lambda_k. 
\end{equation}
This sum depends on the numbering of the vertices only up to a sign, which is $1$ iff the numbering is compatible with the orientation of the simplex.

The span of the forms $\lambda_T$ with $T$ ranging through the $k$-dimensional simplexes in some simplicial complex $\calT$, is denoted  $\WF^k(\calT)$:
\begin{equation}
\WF^k(\calT) = \myspan \{ \lambda_T \ : \ T \in \calT^k \}.
\end{equation} 
Such forms are also called Whitney forms.

For instance given an edge $e$, oriented from the vertex $\dot e$ to the vertex $\ddot e$, the associated Whitney one-form is:
\begin{equation}
\lambda_e = \lambda_{\dot e} \rmd \lambda_{\ddot e} - \lambda_{\ddot e} \rmd \lambda_{\dot e}.
\end{equation}
When $A$ is a general Whitney one-form, we may write:
\begin{equation}
A = \sum_e A_e \lambda_e,
\end{equation}
with:
\begin{equation}
A_e = \int_e A,
\end{equation}
given that for each edge $e$, an orientation has been chosen.

We will also use the following notations. If $x$ and $y$ are two vertices we put:
\begin{equation}
A_{xy} = \begin{cases}
 \phantom{-} A_e & \myif x= \dot e \myand y = \ddot e,\\
 - A_e & \myif x = \ddot e \myand y = \dot e,\\
 \phantom{-} 0&  \quad \textrm{other cases.}
\end{cases}
\end{equation}
In fact this coincides with the definition:
\begin{equation}\label{eq:axy}
A_{xy} = \int_{[xy]} A,
\end{equation}
where $[xy]$ is the oriented edge connecting $x$ to $y$.

\subsection{Parallel transport}
Suppose $[xy]$ is the oriented edge connecting $x$ to $y$, and that $A$ is a constant one form on it. Given a value $u_y$ associated to $y$ we may solve, on $[xy]$, the ordinary differential equation for $u$:
\begin{equation}
\nabla_A u = 0.
\end{equation}
We then have a relation between the vertex values of $u$:
\begin{equation}
u_x = U_{xy} u_y,
\end{equation}
with:
\begin{equation}
U_{xy} = \exp(i A_{xy}),
\end{equation}
given the identity (\ref{eq:axy}).

We refer to $U_{xy}$ as the parallell transport from $y$ to $x$. We may notice:
\begin{equation}\label{eq:ubar}
U_{yx} = U_{xy}^\conj.
\end{equation}
We also use the convention:
\begin{equation}\label{eq:uid}
U_{xx} = 1.
\end{equation}

Discrete gauge transformations are defined as follows. We assume we have a value $\alpha_x\in \bbR$ associated with each vertex $x\in \calT^0$.
It acts on $u$ defined at vertices $x$, with vertex values $u_x$, by the transformation:

\begin{equation}
u_x \mapsto \exp(i \alpha_x) u_x.
\end{equation}

We will always consider that gauge potentials are in $\WF^1(\calT)$ and that gauge transformations are defined by functions $\alpha \in \WF^0(\calT)$, so that $\grad \alpha \in \WF^1(\calT)$.

However the wave functions $u$ will be reconstructed from their vertex values, in several different ways, depending on a constant gauge potential $\tilde A$ on each simplex $T$ (of every dimension). These gauge potentials can in principle be chosen independently of each other. However, in practice, on any simplex $T$, the associated $\tilde A$ will be the average of some globally compatible $A \in \WF^1(\calT) $. For instance, we may then write, when $T$ is a tetrahedron, for  $x  \in T$:
\begin{equation}
A(x) = \tilde A + (1/2) B \times (x - x_T),
\end{equation}
where $B = \curl A$ and $x_T$ is the isobarycentre of $T$.
 
We denote by $X_T[\tilde A]$ the space of complex valued functions $u$ on $T$ such that for every subsimplex $T'$, we have:
\begin{equation}
\Delta_{\tilde A} u = 0.
\end{equation}
 For instance $X_T[0]$ is simply the space of affine functions. Elements of $X_T[\tilde A ]$ are uniquely determined by their vertex values, so that we may consider the nodal interpolator onto $X_T[\tilde A]$.

Given regular meshes $\calT_h$ of size $h$, the associated space of scalar functions will be denoted $X_h[\tilde A]$ (note that $\tilde A$ depends on $h$).

Passing between $X_h[0]$ and $X_h[\tilde A]$ does not produce big errors. In fact we have the following estimates:
\begin{proposition} \label{prop:ltwopi}

Let $\Pi_h$ denote the nodal interpolator onto $X_h[0]$.
For $u \in X[\tilde A]$ we have estimates:
\begin{align}
\| u - \Pi_h u \|_{\rmL^2} &\cleq h^2 \| u\|_{\rmH^1},\label{eq:ltwoerr}\\
\| \nabla u - \nabla \Pi_h u \|_{\rmL^2} &\cleq h \| u\|_{\rmH^1}.
\end{align}
\end{proposition}
\begin{proof}
Consider a simplex $T$ of diameter $1$.
We write, using a Bramble-Hilbert type argument, followed by Proposition \ref{prop:htwobound}:
\begin{align}
\| u - \Pi_h u \|_{\rmL^2(T)} & \cleq \sum_{T' \subcellp T} | u|_{\rmH^2(T')},\\
& \cleq \| A\|_{\rmL^\infty} \| \nabla u\|_{\rmL^2(T)} + \| A\|_{\rmL^\infty}^2 \| u\|_{\rmL^2(T)}.
\end{align}
Then estimate (\ref{eq:ltwoerr}) follows by scaling.

The other estimate is proved similarly.
\end{proof}
In the other direction we note:
\begin{proposition} \label{prop:ltwopit}

Let $\tilde \Pi_h$ denote the nodal interpolator onto $X_h[\tilde A]$.
For $u \in X_h[0]$ we have estimates:
\begin{align}
\| u - \tilde \Pi_h u \|_{\rmL^2} &\cleq h^2 \| u\|_{\rmH^1},\label{eq:ltwoerrt}\\
\| \nabla u - \nabla \tilde \Pi_h u \|_{\rmL^2} &\cleq h \| u\|_{\rmH^1}.
\end{align}
\end{proposition}
\begin{proof}
Consider a simplex $T$ of diameter $1$.
We apply Proposition \ref{prop:honerec}:
\begin{align}
\| u - \tilde \Pi u \|_{\rmL^2(T)} & \cleq \sum_{T' \subcellp T} \| \Delta_{\tilde A} u\|_{\rmL^2(T')},\\
& \cleq  \sum_{T' \subcell T} \| A\|_{\rmL^\infty} \| \nabla u\|_{\rmL^2(T')} + \| A\|_{\rmL^\infty}^2 \| u\|_{\rmL^2(T')},\\
& \cleq \| A\|_{\rmL^\infty} \| \nabla u\|_{\rmL^2(T)} + \| A\|_{\rmL^\infty}^2 \| u\|_{\rmL^2(T)}.
\end{align}
Then (\ref{eq:ltwoerrt}) follows by scaling.

The other estimate is proved similarly.
\end{proof}

\begin{remark}\label{rem:bestt}
From this last Proposition estimates of best approximation in $X_h[\tilde A]$ can be deduced from those in $X_h[0]$, which are well known.
\end{remark}

\subsection{Covariant mass matrix}

We now wish to define $\rmL^2$ scalar products of functions, given their vertex values.  
 
For a tetrahedron $T$, we denote by $M(T)$ the mass matrix on $\WF^0(T)$ with respect to the canonical basis:
\begin{equation}\label{eq:mxy}
M_{xy}(T) = \int_T \lambda_x \lambda_{y}.
\end{equation}
This matrix is real, symmetric and positive definite. The global mass matrix, which may be assembled from the local mass matrices defined above, will be denoted $M$, so that:
\begin{equation}\label{eq:mm}
M_{xy} = \int \lambda_x \lambda_{y}.
\end{equation}

\begin{definition}\label{def:massm}
Given parallel transports $U$ between neighbouring vertices (that is, those connected by an edge), subject to (\ref{eq:ubar}) and (\ref{eq:uid}), we define a \emph{covariant $\rmL^2$-product}, as follows. Given complex scalar fields $u$ and $v$, with well defined vertex values, we set:
\begin{equation}\label{eq:covscal}
\langle u, v \rangle_U = \sum_{xy} u_x^\conj M_{xy} U_{xy}v_y,
\end{equation}
where the matrix $M$ is the mass matrix already defined in (\ref{eq:mm}).
\end{definition}

This matrix may also be assembled from terms local to each tetrahedron.

\begin{proposition}\label{prop:herminv}
This covariant scalar product (\ref{eq:covscal}) is hermitian and gauge invariant, under transformations :
\begin{align}
u_x & \mapsto \exp(i \alpha_x) u_x,\\
U_{xy} & \mapsto  \exp(i \alpha_x)  U_{xy} \exp(- i \alpha_y),\\
v_y & \mapsto \exp(i \alpha_y) v_y.
\end{align}

\end{proposition}
\begin{proof}
(i) Hermitian:
\begin{align}
\langle v, u \rangle_U & = \sum_{xy} v_x^\conj M_{xy} U_{xy}u_y,\\
& = \sum_{xy}  u_y M_{yx} U_{yx}^\conj v_x^\conj,\\
& = \langle u, v \rangle_U^\conj .
\end{align}

(ii) Gauge invariance is trivial.
\end{proof}

\begin{proposition}\label{prop:covhtwo}
The covariant scalar product (\ref{eq:covscal}) is $h^2$-conforming with respect to the $\rmH^1$ norm on the space $X_h[0]$ associated with a regular mesh of width $h$. In other words, for $u,v \in X_h[0]$ we have:
\begin{equation}
|\langle u, v \rangle - \langle u, v \rangle_U | \cleq h^2 \| u \|_{\rmH^1} \| v \|_{\rmH^1}.
\end{equation}

\end{proposition}
\begin{proof}
We work on a simplex $T$ of diameter $h$.

We want to estimate the error:
\begin{equation}
\langle u, v \rangle - \langle u, v\rangle_U,
\end{equation}
And we do this by decomposing $u$ and $v$ according to:
\begin{align}
u(x) & = u_T + u_T^\prime \cdot (x - x_T), \mywith u_T^\prime = \grad u, \\
v(x) & = v_T + v_T^\prime \cdot (x - x_T), \mywith u_T^\prime = \grad v,
\end{align} 
where $u_T$ denotes the value of $u$ at the isobarycenter of $T$, which is denoted $x_T$, and $u_T^\prime$ is the gradient of $u$. Likewise for $v$.

We now treat the four terms this decomposition gives:

(i) We have:
\begin{align}
\langle u_T, v_T \rangle - \langle u_T, v_T \rangle_U & = \sum_{xy} u_T^\conj M_{xy} (U_{xy} -1) v_T,\\
& = \sum_{xy} u_T^\conj M_{xy} (\cos (A_{xy})  -1) v_T,\\
& \cleq  h^2 \|  u \|_{\rmL^2(T)}   \|  u \|_{\rmL^2(T)}.   
\end{align}
The trick was to symmetrize in $x$ and $y$, use (\ref{eq:ubar}) and $|\cos (A_{xy})  -1| \cleq h^2$.

(ii) We have:
\begin{align}
 & \langle u_T, v_T^\prime \cdot (y - y_T) \rangle - \langle u_T, v_T^\prime \cdot (y - y_T) \rangle_U,\\
= &\sum_{xy} u_T^\conj M_{xy} (U_{xy} -1) ( v_T^\prime \cdot (y - y_T)),\\
\cleq & h^2   \|  u \|_{\rmL^2(T)}   \|  \grad v \|_{\rmL^2(T)}.
\end{align}
Here we combined $|U_{xy} -1| \cleq h$ and $|y-y_T| \leq h$.

(iii) There is a similar term where $u$ and $v$ exchange roles.

(iv) Finally, the last term, involving two gradients, yields a factor $h^3$.

This completes the proof.
\end{proof}

\begin{proposition} \label{prop:covscalt}
The covariant scalar product (\ref{eq:covscal}) is $h^2$-conforming with respect to the $\rmH^1$ norm on the space $X_h[\tilde A]$ associated with a regular mesh of width $h$. In other words, for $u,v \in X_h[\tilde A]$:
\begin{equation}
|\langle u, v \rangle - \langle u, v \rangle_U | \cleq h^2 \| u \|_{\rmH^1} \| v \|_{\rmH^1}.
\end{equation}
\end{proposition}

\begin{proof}
We let $\Pi_h$ denote the nodal interpolator onto $X_h[0]$. 
We have, by Proposition \ref{prop:ltwopi}:
\begin{equation}
|\langle u, v \rangle - \langle \Pi_h u, \Pi_h v \rangle | \cleq h^2 \| u \|_{\rmH^1} \| v \|_{\rmH^1},
\end{equation}
and, by Proposition \ref{prop:covhtwo}:
\begin{equation}
|\langle \Pi_h u, \Pi_h v \rangle - \langle \Pi_h u, \Pi_h v \rangle_U | \cleq h^2 \| \Pi_h u \|_{\rmH^1} \| \Pi_h v \|_{\rmH^1}.
\end{equation}
We conclude by the stability of $\Pi_h: X_h[\tilde A] \to X_h[0]$ in $\rmH^1$ norm, which can be deduced from Proposition \ref{prop:ltwopi}.
\end{proof}

\subsection{Covariant stiffness matrix}

\begin{lemma}\label{lem:nabexp}

La $A$ be a constant one-form on the edge $[xy]$. We consider functions $u:[xy] \to \bbC$ satisfying $\Delta_A u = 0$. Then:
\begin{align}
(\nabla_A u (x)) (y-x) & = u(y) \exp(i A_{xy}) - u(x),\label{eq:naux}\\
(\nabla_A u (y)) (y-x) & = u(y) - \exp(- i A_{xy}) u(x).\label{eq:nauy}
\end{align}

\end{lemma}

\begin{proof}
We parametrize the edge linearly from $0$ to $1$, with a variable $t$. The solutions  to $\Delta_A u = 0$ have the form:
\begin{equation}
u(t) = (a + b t) \exp(- i A_{xy} t).
\end{equation}
We get: 
\begin{align}
u(0) & = a,\\
u(1) & = (a + b ) \exp(-i A_{xy}).
\end{align}
We also compute:
\begin{equation}
(\nabla + i A)  u (t) = \exp(- i A_{xy} t) b.
\end{equation}
We deduce:
\begin{align}
(\nabla + i A) u (0) &= b,\\
& = u(1) \exp(i A_{xy}) - u(0).
\end{align}
and likewise:
\begin{align}
(\nabla + i A) u (1) &= \exp(- i A_{xy}) b,\\
& = u(1) - \exp(- i A_{xy}) u(0).
\end{align}
From this the lemma follows.
\end{proof}

We use these identities as follows. Given a mesh and a Whitney one-form $A$ on it, we define, on every simplex, $\tilde A$, to be the average of $A$. We remark that on edges $\tilde A = A$. Associated with $\tilde A$ we have the space $X[\tilde A]$ of scalar functions, on the mesh,  determined by imposing $\Delta_{\tilde A} u =0 $  on every simplex (of every dimension). Then for $u \in X[\tilde A]$, the vertex values of $\nabla_A u$ are computable, using (\ref{eq:naux}, \ref{eq:nauy}). Thus we may interpolate it onto affine vector fields in a computable way. We now detail how this leads to a numerical method.

Let $T$ be a tetrahedron. We consider, at each vertex $x$ and for each edge $[xy]$ emanating from $x$, the tangent vector $\tau_{xy}= y-x$. At any vertex $x$, the three tangent vectors, pointing to the three other vertices of the tetrahedron, constitute a basis of $\bbR^3$. We denote by $\mu_{xy}$ the dual basis, i.e.:
\begin{equation}
\mu_{xz}(\tau_{xy})  = \begin{cases}
1 & \myif y=z,\\
0 & \myif y\neq z.
\end{cases}
\end{equation}

Given nodal values for a scalar function $u$ on $T$, we construct an affine one-form $v$ by setting for each pair of vertices $(x, y)$:
\begin{equation}\label{eq:vecdof}
v_{xy} = v(x) (\tau_{xy}) = U_{xy} u(y) - u(x).
\end{equation}
and, summing over pairs of vertices we define the vectorfield:
\begin{equation}
v = \sum_{xy} v_{xy}  \lambda_x \mu_{xy}.
\end{equation}
By Lemma \ref{lem:nabexp}, if $u\in X[\tilde A]$,  $v$ is then the affine vectorfield on $T$ coinciding with $\nabla_A u$ at the vertices.

The $\rmL^2$ scalar product on affine one forms (or vector fields)  can be expressed with the scalar mass matrix, as follows. Let $v$ and $v'$ be affine one forms. Define the numbers $v_{xy}$ by (\ref{eq:vecdof}), and proceed similarly for and $v'$. 
\begin{equation}
\int v^\conj \cdot v' = \sum_{xy, zt} v_{xy}^\conj v'_{zt} \ \mu_{xy} \cdot \mu_{zt} M_{xz}.
\end{equation}
The covariant scalar product on affine one forms is defined by setting:
\begin{equation}
\langle v,  v' \rangle_U = \sum_{xy, zt} v_{xy}^\conj U_{xz} v'_{zt} \ \mu_{xy} \cdot \mu_{zt} M_{xz}.
\end{equation}

\begin{definition}\label{def:stiffm}
Let $u$ and $v$ be functions with well defined vertex values. We define a modified bilinear form, as a sum over tetrahedra:
\begin{equation}
\tilde a(u,v) = \sum_T \tilde a_T(u,v),
\end{equation}
where the contribution of tetrahedron $T$ is:
\begin{equation} \label{eq:stiffcontt}
\tilde a_T(u,v) = \sum_{xy, zt} (U_{xy}u_y -u_x ) ^\conj U_{xz} (U_{zt} u_t - u_z)  \mu_{xy} \cdot \mu_{zt} M_{xz}(T).
\end{equation}
\end{definition}
A crucial identity, which follows from Lemma \ref{lem:nabexp}, is that for $u,v \in X[\tilde A]$:
\begin{equation}\label{eq:stiffdisc}
\tilde a(u,v) = \langle \Pi \nabla_A u,  \Pi \nabla_A v \rangle_U. 
\end{equation}
The interpolation operator $\Pi$  appearing here, is the nodal interpolation onto affine one forms.

\begin{proposition} 
The discrete stiffness matrix defined by (\ref{eq:stiffcontt}) is gauge invariant (in the same sense as in Proposition \ref{prop:herminv}) and hermitian.
\end{proposition}
\begin{proof}
Trivial.
\end{proof}

Next we examine the consistancy of the discrete stiffness matrix. The key estimate is the following one:
\begin{proposition} \label{prop:piafff}
Let $\Pi_h$ denote interpolation onto affine one forms.
Choose $u \in X_h[\tilde A]$. We have :
\begin{equation}
\| \nabla_A u - \Pi_h \nabla_A u\|_{\rmL^2} \cleq h^2 \| u\|_{\rmH^1}.
\end{equation}

\end{proposition}

\begin{proof}
We work first on a simplex $T$ of diameter one. We have an estimate:
\begin{equation}
\| \nabla_A u - \Pi_h \nabla_A u\|_{\rmL^2(T)}  \cleq \sum_{T' \subcellp T} \| \Delta \nabla_A u \|_{\rmL^2(T')}.
\end{equation}

Consider now a face $T'$ of $T$. We write:
\begin{equation}
\Delta \nabla_A u  = \Delta_{\tilde A} \nabla_A u  +  |\tilde A|^2 \nabla_A u - 2 i \tilde A \cdot \nabla \nabla_A u. \label{eq:estimatedelta}
\end{equation} 
We consider the three terms on the right hand side.

(i) First term:
\begin{align}
\Delta_{\tilde A} \nabla_A u & = \nabla_A \Delta_{\tilde A} u  + \Delta_{\tilde A} (i A u),\\
&= \Delta_{\tilde A} (i A u), \label{eq:laststop}
\end{align}
since $u \in X[\tilde A]$. Suppose more generally that $v$ is any affine function, and that $u$ is arbitrary. Then:
\begin{equation}
\Delta_{\tilde A} (v u) = v (\Delta_{\tilde A} u) + 2 \nabla v \cdot \nabla u + 2 i (\tilde A \cdot \nabla v) u.
\end{equation}
Letting the role of $v$ be played by $A$,  we may continue from (\ref{eq:laststop}):
\begin{align}
\| \Delta_{\tilde A} (i A u)  \|_{\rmL^2} \cleq \|B\|_{\rmL^\infty} \|\nabla u \|_{\rmL^2} + \| \tilde A\|_{\rmL^\infty} \| B \|_{\rmL^\infty} \| u\|_{\rmL^2} .
\end{align}

(ii) Second term:
\begin{equation}
\| \, |\tilde A|^2 \nabla_A u \|_{\rmL^2} \leq \| \tilde A \|_{\rmL^\infty}^2 \| \nabla u\|_{\rmL^2} + \| \tilde A \|_{\rmL^\infty}^2 \| A \|_{\rmL^\infty} \| u\|_{\rmL^2}.
\end{equation}

(iii) Third term:
\begin{equation}
\| \tilde A \cdot \nabla \nabla_A u \|_{\rmL^2} \cleq \| \tilde A\|_{\rmL^\infty}  \| \hess u \|_{\rmL^2} + \| \tilde A \|_{\rmL^\infty} \| B \|_{\rmL^\infty} \| u \|_{\rmL^2} + \| \tilde A \|_{\rmL^\infty} \| A \|_{\rmL^\infty} \| \nabla u\|_{\rmL^2}.
\end{equation}

Summing the three terms, we get, on $T'$ :
\begin{align}
\| \Delta \nabla_A u\|_{\rmL^2} \leq&  \| A\|_{\rmL^\infty}  \| \hess u \|_{\rmL^2} +(  \| A \|_{\rmL^\infty}^2  + \|B\|_{\rmL^\infty}) \| \nabla u\|_{\rmL^2} +\\
& ( \| A \|_{\rmL^\infty}^3 +\|  A\|_{\rmL^\infty} \| B \|_{\rmL^\infty})  \| u \|_{\rmL^2}.
\end{align}

Using Propositions \ref{prop:ltwobound}, \ref{prop:gradbound}, \ref{prop:htwobound}, the right hand side terms on $T'$ may be bounded by terms attached to $T$, as follows:
\begin{align}\label{eq:keyest}
\| \nabla_A u - \Pi \nabla_A u\|_{\rmL^2(T)}  \cleq & (  \| A \|_{\rmL^\infty}^2  + \|B\|_{\rmL^\infty}) \| \nabla u\|_{\rmL^2(T)} + \\
& ( \| A \|_{\rmL^\infty}^3 +\|  A\|_{\rmL^\infty} \| B \|_{\rmL^\infty})  \| u \|_{\rmL^2(T)}.
\end{align}

Finally one concludes by scaling, noting that $A$ scales like a one-form and $B$ like a two-form.

\end{proof}

We also notice the variant:
\begin{proposition}\label{prop:piafffvar}
Let $\Pi_h$ denote interpolation onto affine one forms.
Choose $u \in X_h[\tilde A]$. We have :
\begin{equation}
\| \nabla  \nabla_A u - \nabla \Pi_h \nabla_A u\|_{\rmL^2} \leq h \| u\|_{\rmH^1}.
\end{equation}
\end{proposition}
\begin{proof}
We go through the preceding proof to obtain the following variant of (\ref{eq:keyest}):
\begin{align}
\| \nabla \nabla_A u - \nabla \Pi_h \nabla_A u\|_{\rmL^2(T)}  \cleq & (  \| A \|_{\rmL^\infty}^2  + \|B\|_{\rmL^\infty}) \| \nabla u\|_{\rmL^2(T)} + \\
& ( \| A \|_{\rmL^\infty}^3 +\|  A\|_{\rmL^\infty} \| B \|_{\rmL^\infty})  \| u \|_{\rmL^2(T)}.
\end{align}
Then the scaling gives the factor $h$ this time.
\end{proof}

\begin{proposition} \label{prop:stiffterror} For $u, v$ in $X_h[\tilde A]$ we have an estimate:
\begin{equation}
|\int (\nabla_A u)^\conj \cdot \nabla_A v - \tilde a(u,v) | \cleq h^2 \|u\|_{\rmH^1} \| v\|_{\rmH^1}.
\end{equation}
\end{proposition}
\begin{proof}
We first use Proposition \ref{prop:piafff} to get:
\begin{equation}
|\int (\nabla_A u)^\conj \cdot \nabla_A v    -  \langle \Pi_h \nabla_A u,  \Pi \nabla_A v \rangle | \cleq h^2 \|u\|_{\rmH^1} \| v\|_{\rmH^1}.
\end{equation}
Then we use Proposition \ref{prop:covhtwo} (generalised from scalar to vector fields):
\begin{equation}\label{eq:tob}
| \langle \Pi_h \nabla_A u,  \Pi_h \nabla_A v \rangle -  \langle \Pi_h \nabla_A u,  \Pi_h \nabla_A v \rangle_U |  \cleq h^2 \| \Pi_h \nabla_A u \|_{\rmH^1} \| \Pi_h \nabla_A v \|_{\rmH^1}.
\end{equation}
The rest of the proof is devoted to bounding the right hand side of this estimate.

Using Proposition \ref{prop:piafff} we get:
\begin{align}
\| \Pi_h \nabla_A u \|_{\rmL^2(T)} & \cleq \|  \nabla_A u \|_{\rmL^2(T)}  + h^2 \|u\|_{\rmH^1(T)},\\
&\cleq \| u \|_{\rmH^1(T)}. \label{eq:honefirst}
\end{align}
Using Proposition \ref{prop:piafffvar} we get:
\begin{align}
\| \nabla \Pi_h \nabla_A u \|_{\rmL^2(T)} \cleq \|  \nabla \nabla_A u \|_{\rmL^2(T)}  + h \|u\|_{\rmH^1(T)}.
\end{align}
On a simplex of diameter $1$ we have, using Proposition \ref{prop:htwobound}:
\begin{align}
\|  \nabla \nabla_A u \|_{\rmL^2(T)}  & \cleq \|  \nabla \nabla u \|_{\rmL^2(T)}  + \| A\|_{\rmL^\infty} \| \nabla u\|_{\rmL^2(T)} + \| B\|_{\rmL^\infty} \| u\|_{\rmL^2(T)} ,\\
& \cleq  \| A\|_{\rmL^\infty}  \| \nabla u\|_{\rmL^2(T)} +( \| A\|_{\rmL^\infty}^2 +  \| B\|_{\rmL^\infty})  \| u\|_{\rmL^2(T)} .
\end{align}
Since the left and the right hand side scale the same way, we deduce:
\begin{equation}
\| \nabla \Pi_h \nabla_A u \|_{\rmL^2(T)} \cleq \|u\|_{\rmH^1(T)}. \label{eq:honesec}
\end{equation}

From (\ref{eq:honefirst}) and (\ref{eq:honesec}) we may conclude:
\begin{equation}
\|  \Pi_h \nabla_A u \|_{\rmH^1(T)} \cleq \|u\|_{\rmH^1(T)}.
\end{equation}

Inserting this in (\ref{eq:tob}) completes the proof.
\end{proof}

\subsection{Conclusions}

Summing up, the situation is as follows.

Consider the eigenvalue problem: Find $u\in \rmH^1_0(S)$ and $E \in \bbR$ such that for all $v \in \rmH^1_0(S)$:
\begin{equation}\label{eq:eigcont}
a(u,v) = E \langle u, v \rangle.
\end{equation}
Since we assume that the domain $S$ is convex and that the gauge potential $A$ is smooth, elliptic regularity holds, in the sense that the solution operator for $\Delta_A$ maps $\rmL^2(S)$ to $\rmH^1_0(S) \cap \rmH^2(S)$. 

We do the Galerkin formulation on the space $X_h[\tilde A]$ attached to a mesh of width $h$. The order of convergence for discrete eigenvectors is $h$ in $\rmH^1$ norm and $h^2$ in $\rmL^2$ norm, since these are the orders of best approximation, see Remark \ref{rem:bestt}. The order of convergence of the eigenvalue is deduced to be $h^2$ (see in particular Lemma 3.1 in \cite{BabOsb89}).

Then we consider the modified formulation: Find $\tilde u \in X_h[\tilde A]$ and $\tilde E$ such that for all $\tilde v\in X_h[\tilde A]$:
\begin{equation}\label{eq:eigdisct}
\tilde a(\tilde u, \tilde v) = \tilde E \langle \tilde u, \tilde v \rangle_U,
\end{equation}
where the modified bilinear form $\tilde a$ was defined in Definition \ref{def:stiffm} whereas $\langle \cdot , \cdot \rangle_U$ was defined in Definition \ref{def:massm}. These modifications produce an error of order $h^2$ in $\rmH^1$ norm, as was shown in Propositions \ref{prop:stiffterror} and \ref{prop:covscalt}. Therefore the preceding orders of convergence for eigenvectors and eigenvalues are maintained (for the eigenvalue see Lemma 5.1 in \cite{BanOsb90}).

Finally we may interpolate the eigenvector $\tilde u \in X_h[\tilde A]$ onto $X_h[0]$ and still get the same orders of convergence, using Lemma \ref{prop:ltwopit}.

Of course we could also consider that we do the variational formulation of (\ref{eq:eigcont}) on $X_h[0]$ (rather than $X_h[\tilde A]$) and that the discretization (\ref{eq:eigdisct}), where $\tilde a$ is defined explicitely from vertex values in (\ref{eq:stiffcontt}), constitutes a variational crime on $X_h[0]$. However for elements $u,v \in X_h[0]$ the formula (\ref{eq:stiffdisc}), which is essential to our analysis, is then no longer true. In fact we expect the error of consistency in $\rmH^1$ norm to be $h$ in this interpretation, whereas it was $h^2$ in the preceding one. This would ruin the analysis, even though we have described the \emph{same} numerical method.

Finally it must be noted that the above analysis requires $A$ to be a Whitney form on each grid. There is an implicit step where, given $A$ on $S$ one first approximates it on the grid $\calT_h$, to be able to use the previous analysis. In general this step produces an error of order $h$ for smooth $A$, jeopardizing the above analysis. For details on how order $h$ can be obtained, under weaker hypotheses,  see \cite{ChrHal11SINUM}\cite{ChrHal12JMP}. But in the important case of a constant magnetic field $B$, the magnetic vectorpotential has (globally) the form:
\begin{equation}
A(x) = A_0 + (1/2) B \times x,
\end{equation}
for some constant $A_0$. Then the approximation of $A$ by Whitney forms is exact and our analysis applies. We consider this case to be sufficiently important to justify our new method.

\section{Extension to the Pauli equation \label{sec:pauli}}

A Pauli wave function is a complex valued two-component spinor $\psi$ in $\rmH^1_0(S)\otimes\bbC^2$. We write:
\begin{equation}
\psi =\left( \begin{array}{c}
\psi_0 \\
\psi_1
\end{array}\right).
\end{equation}
Let $\vec\sigma = (\sigma_1, \sigma_2, \sigma_3)$ be the hermitian and unitary Pauli matrices collected in a vector. The components are:
\begin{equation}
\sigma_1 = \left(\begin{array}{cc}
0 & 1 \\
1 & 0 
\end{array}\right),
\quad
\sigma_2 = \left(\begin{array}{cc}
0 & -i \\
i & 0 
\end{array}\right),
\quad
\sigma_3 = \left(\begin{array}{cc}
1 & 0 \\
0 & -1 
\end{array}\right).
\end{equation}
In natural units the time-dependent Pauli equation reads:
\begin{equation}
-\frac{1}{2}(\vec\sigma\cdot \nabla_A)^2\psi = i \partial_V \psi,
\end{equation}
where $\nabla_A\psi = \grad\psi + iA\psi$ is the covariant spatial gradient of $\psi$ and $\partial_V \psi = \dot \psi + iV\psi$ is the covariant time derivative. If we assume that the time-dependence is $\psi(x,t) = \psi(x)e^{-iEt}$, where $E\in \mathbb R$, then we get the Pauli eigenvalue equation
\begin{equation}
-\frac{1}{2} (\vec\sigma\cdot \nabla_A)^2\psi + V\psi = E\psi.
\end{equation}
From the identity
\begin{equation}
\sigma_i\sigma_j = \delta_{ij} \bbI + i\sum_k \varepsilon_{ijk}\sigma_k,
\end{equation}
where $\varepsilon$ is the totally antisymmetric Levi-Civita symbol with $\varepsilon_{123}=1$, we can rewrite the equation as
\begin{equation}
- (\nabla_A^2 + \vec\sigma\cdot\curl A)\psi + V\psi = E\psi.
\end{equation}

The variational formulation of the Pauli eigenvalue problem consists in finding $\psi \in (H^1_0(S)\otimes\mathbb C^2)$ and $E \in  \bbR$, $\psi\neq 0$, such that for all $\phi\in \rmH^1_0(S)\otimes \bbC^2$ : 
\begin{equation}\label{cont:pauli_evp}
a(\psi, \phi) + b(\psi,\phi) + c(\psi, \phi) = E\langle \psi, \phi\rangle,
\end{equation}
where $a(\cdot,\cdot)$, $b(\cdot,\cdot)$, and $c(\cdot, \cdot)$ are the bilinear forms given by:
\begin{equation}
\begin{split}
a(\psi, \phi) &= \langle \nabla_A\psi, \nabla_A\phi\rangle, \\
b(\psi, \phi) &= \langle V\psi, \phi\rangle, \\
c(\psi, \phi) &= -\langle (\vec\sigma\cdot\curl A)\psi,\phi\rangle.
\end{split}
\end{equation}

Equation \eqref{cont:pauli_evp} remains invariant under gauge transformations (\ref{eq:gaugetransfa}, \ref{eq:gaugetransfu}). We proceed to define gauge-invariant discretizations.

The bilinear forms $a$ and $\langle \cdot, \cdot \rangle$ are discretized as before, and the analysis carries over straighforwardly to spinors. 

The forms $b$ and $c$ are treated in analogy with the previously defined covariant scalar product.

Explicitely we define:
\begin{equation}\label{eq:tildeb}
\tilde b(\psi, \phi) = \sum_{xy} (\int \lambda_x V \lambda_y) \psi_x^\conj  \cdot U_{xy} \phi_y.
\end{equation}
If $V$ is a smooth potential on $S$ one can approximate it by piecewise linears, to second order, before plugging it in to  (\ref{eq:tildeb}), as was done in \cite{ChrHal11SINUM}.

We also define:
\begin{equation}\label{eq:tildec}
\tilde c(\psi, \phi) = - \sum_T \sum_{xy}  (\int_T \lambda_x \lambda_y) (\vec\sigma\cdot\curl A ) \psi_x^\conj  \cdot U_{xy} \phi_y.
\end{equation}
Here we rely on the fact that $\vec\sigma\cdot\curl A$ is a constant matrix on each $T$. 

Since there are no novelties in the proofs, we just state the conclusions: The above discretization technique yields second order convergence under the same hypotheses as before, which include constant magnetic fields and smooth scalar potentials. Lower order convergence rates can be obtained under weaker hypotheseses, as in \cite{ChrHal11SINUM}.

\section*{Acknowledgement}
The authors were supported by the European Research Council through the FP7-IDEAS-ERC Starting Grant scheme, project 278011 STUCCOFIELDS.

\bibliography{../Bibliography/alexandria,../Bibliography/newalexandria,../Bibliography/mybibliography}{}
\bibliographystyle{plain}

\end{document}

%% file: definitionsSym.tex
\DeclareMathOperator{\Div}{div}

\DeclareMathOperator{\grad}{grad}
\DeclareMathOperator{\curl}{curl}

\DeclareMathOperator{\hess}{hess}
\DeclareMathOperator{\myspan}{span}

\newcommand{\rec}{\mathrm{rec}}

\newcommand{\conj}{\dagger}

\renewcommand{\div}{\Div}

\newcommand{\rmd}{\mathrm{d}}

\newcommand{\rmH}{\mathrm{H}}
\newcommand{\rmL}{\mathrm{L}}

\newcommand{\bbC}{\mathbb{C}}

\newcommand{\bbI}{\mathbb{I}}

\newcommand{\bbR}{\mathbb{R}}

\newcommand{\calA}{\mathcal{A}}

\newcommand{\calT}{\mathcal{T}}

\newcommand{\beq}{\begin{equation}}
\newcommand{\eeq}{\end{equation}}

\newcommand{\WF}{\mathfrak{W}}

\newcommand{\myand}{ \quad \textrm{ and } \quad }

\newcommand{\myif}{\textrm{ if }}
\newcommand{\with}{\textrm{ with }}
\newcommand{\mywith}{\with}

\newcommand{\cleq}{\preccurlyeq}

\newcommand{\subcell}{\lhd} 
\newcommand{\subcelln}{\lhd^\cdot }
\newcommand{\subcellp}{\lhd_\cdot }



%% file: definitionsEnv.tex
\newtheorem{theorem}{Theorem}[section]
\newtheorem{lemma}[theorem]{Lemma}

\newtheorem{proposition}[theorem]{Proposition}

\theoremstyle{definition}
\newtheorem{definition}{Definition}[section]

\theoremstyle{remark}
\newtheorem{remark}{Remark}[section]

\theoremstyle{plain}